\documentclass[12pt,twoside]{paper}   
\usepackage{graphicx,fancyhdr,url,amsfonts} 
\usepackage{hyperref,amssymb,amsmath,color,amsthm}
\newtheorem{theorem}{Theorem}

\newtheorem{proposition}{Proposition}
\newtheorem{remark}{Remark}

\begin{document}   

\fancyhead[RO]{P.B. Acosta--Hum\'anez,  \'Alvarez--Ram\'{\i}rez, T.J. Stuchi}
\fancyhead[LE]{Exceptional potentials via  polynomial bi-homogeneous potentials}
\fancyhead[RE,LO]{\thepage}

 \author{Primitivo B. Acosta--Hum\'anez\thanks{Instituto de Matem\'atica, Facultad de Ciencias, Universidad Aut\'onoma de Santo Domingo, Dominican Republic.  (pacosta-humanez@uasd.edu.do), (ORCID ID: \href{https://orcid.org/0000-0002-5627-4188}{0000-0002-5627-4188}). $^*$Corresponding Author.},\\ Martha \'Alvarez--Ram\'{\i}rez\thanks{Departamento de Matem\'aticas, UAM--Iztapalapa, M\'exico,
City, M\'exico. (mar@xanum.uam.mx), (ORCID ID: \href{https://orcid.org/0000-0001-9187-1757}{0000-0001-9187-1757})},\\ Teresinha J. Stuchi\thanks{Instituto de F\'isica, Universidade Federale de Rio de Janeiro, Rio de Janeiro, Brazil (tstuchi@if.ufjr.br), (ORCID ID: \href{https://orcid.org/0000-0003-0700-6205}{0000-0003-0700-6205}) }}

 \title{A note on the  integrability of  exceptional potentials via  polynomial bi-homogeneous potentials}
 
 \shortauthor{P.B. Acosta--Hum\'anez,  \'Alvarez--Ram\'{\i}rez, T.J. Stuchi}

\maketitle

 \begin{abstract}
This paper is concerned with the polynomial integrability of the  two-dimensional Hamiltonian systems associated to complex homogeneous polynomial potentials  
of degree $k$ of type  $V_{k,l}=\alpha (q_2-i q_1)^l (q_2+iq_1)^{k-l}$    with $\alpha\in\mathbb{C}$ and $l=0,1,\dots, k$, called exceptional potentials. Hietarinta \cite{Hietarinta1983} proved that the potentials with   $l=0,1,k-1,k$ and $l=k/2$ for $k$ even are polynomial integrable.
We present  an  elementary proof of this fact in the context of the polynomial bi-homogeneous potentials, as was introduced by 
Combot et al. \cite{Combot2020}.
In addition, we take advantage of the fact that we can exchange the exponents  
to derive an additional   first integral   for $V_{7,5}$,  unknown  so far. The paper concludes with   a Galoisian analysis  for $l=k/2$.
\end{abstract}
 

\footnotesize{\noindent\textbf{Resumen}\\

\noindent Este art\'iculo trata de la integrabilidad polinomial de sistemas hamiltonianos potenciales polinomiales homogeneos complejos de grado $k$ del tipo  $V_{k,l}=\alpha (q_2-i q_1)^l (q_2+iq_1)^{k-l}$ con $\alpha\in\mathbb{C}$ y $l=0,1,\dots, k$, denominados potenciales excepcionales. Hietarinta \cite{Hietarinta1983} prob\'o que los potenciales con $l=0,1,k-1,k$ y $l=k/2$ para $k$ par son integrables polinomialmente.
Presentamos una prueba elemental de este hecho en el contexto de los potenciales bi-homog\'eneos polinomiales, tal como fueron introducidos por Combot et al. \cite{Combot2020}. Adicionalmente, tomamos ventaja del hecho que podemos intercambiar los exponentes para obtener una integral primera adicional para $V_{7,5}$,  desconocido hasta donde sabemos. El art\'iculo finaliza con un an\'alisis Galoisiano para $l=k/2$.\\

\noindent \textbf{Palabras claves:} Sistemas hamiltonianos de dos grados de libertad, potenciales homog\'eneos, integrabilidad.
}

\smalltableofcontents
\section{Introduction}
The existence of first integrals in Hamiltonian systems is of utmost importance, since it allows us to lower the dimension of such systems.
The Liouville theorem states that a Hamiltonian system with $n$ degrees of freedom is completely integrable if
there exist $n$ independent first integrals in involution. Thus, the problem can be integrated  by quadratures, see \cite{Abraham}. In other words,
completely integrable Hamiltonians are systems for which, in a way or another, we can obtain a full description of their solutions

The study of the integrability of Hamiltonian systems with homogeneous potentials  
within the framework of Ziglin's  theory  \cite{zig1,zig2}
was started by Yoshida  \cite{yo1,yo2}. Since then,  a great deal of research has been done, mainly based on  differential Galois theory of ordinary linear differential equations, where the Morales-Ramis theorem plays a key role  giving  necessary conditions for the  integrability in the Liouville sense of a Hamiltonian system   having meromorphic first integrals.  See  \cite{mora1, Maciejewski2005, Studzi2013}  and references therein.

 We assume that  $\mathbb{C}^{2n}$ is the  symplectic linear space with  $\mathbf{q} = (q_1,\dots, q_n)$ and $\mathbf{p} = (p_1,\dots, p_n)$  being  the position and momentum vectors, respectively.
Suppose we have a Hamiltonian system with   Hamiltonian function
\begin{equation}\label{ham_1n}
H= \frac{1}{2}\sum_{i=1}^n p_i^2 +  V(\mathbf{q}),
\end{equation}
where   the potential energy  $V=V(\mathbf{q})\in  \mathbb{C}[\mathbf{q}]$ is  a homogeneous polynomial of degree $k$, that is, $V(\lambda \mathbf{q})=\lambda^kV(\mathbf{q})$ for all nonzero  $\lambda\in \mathbb{C}$ and  $k\in \mathbb{Z}$.
The motion equations of   the above Hamiltonian can be written as:
\begin{equation}\label{ham_2n}
\frac{d }{dt} \mathbf{q} = \mathbf{p}, \qquad  \frac{d }{dt} \mathbf{p} = - V'(\mathbf{q}),
\end{equation}
where $V'(\mathbf{q}) $ indicates the gradient of $V(\mathbf{q})$.
The concept of Darboux points play a central role in the integrability analysis of a homogeneous potential.
A non-zero point $\mathbf{d}\in \mathbb{C}^n$ is called a \emph{Darboux point} of  \eqref{ham_2n}
if it is a solution of  $V'(\mathbf{q}) =\lambda\mathbf{d}$,
where $\lambda\in\mathbb{C}\setminus\{0\}$.  It is also said  that a Darboux point $\mathbf{d}$ is a {\em proper} Darboux point if and only if
$V'(\mathbf{d})\neq \mathbf{0}$. Otherwise, it is called an improper Darboux point.
In other words,  a Darboux point can be seen as  a point  in the complex projective space $\mathbb{C}\mathbb{P}^ {n-1}$ with
$[d_1:\cdots : d_n]$ as homogeneous coordinates. For more information on this subject the reader is referred to \cite{Maciejewski2005}.

From now on, we shall constrain ourselves to the   Hamiltonian systems with two degrees of freedom.
At this stage, we consider the group of  complex matrices  $A_{2\times 2} \in {\rm PO}(2,\mathbb{C})\cong \mathbb{C}\mathbb{P}^1$. Here ${\rm PO}(2,\mathbb{C})$ denotes the group of $2\times 2$ complex matrices $A$ such that
$AA^T = \alpha I_2$, where $\alpha\in \mathbb{C}\setminus\{0\}$ and
$I_2$ is the $2$-dimensional identity matrix. The potentials $V_1({\mathbf q})$ and $V_2({\mathbf q})$
are said to be equivalents if  there exists a matrix $A\in {\rm PO}(2,\mathbb{C})$ such that $V_1({\mathbf q}) = V_2(A {\mathbf q})$.
This concept  was inspired by a Hietarinta's result proved  in \cite{Hietarinta1987}, namely that  if $V_1({\mathbf q})$ is integrable, then $V_2(A{\mathbf q})$ is also integrable for any $A\in {\rm PO}(2,\mathbb{C})$.

A homogeneous polynomial  $V\in \mathbb{C} [q_1,q_2]$ of degree $k$ is a polynomial such that
$ V(q_1,q_2)=q^k_1 V(1,q_2/q_1)= q_2^k V(q_1/q_2,1)$.
Let be $v(x):=V(1,x)$ and $w(\xi):=V(\xi ,1)$.
Since  ${\rm PO}(2,\mathbb{C}) =\mathbb{C}\mathbb{P}^1$,  we use $[q_1:q_2]$   as   homogeneous coordinates on $\mathbb{C}\mathbb{P}^1$.  Moreover, points $[1:q_2]\in \mathbb{C}\mathbb{P}^1$ will be identified
with  the affine part of  $\mathbb{C}\mathbb{P}^1$ when it is parameterized  by coordinates $x=q_2/q_1$, $q_1\neq 0$.
The points in a neighborhood of infinity  are parameterized by the other coordinate   $\xi = q_1/q_2$.
In \cite{Maciejewski2005}, the authors  choose   amid all equivalent points  a representative $V$,  whose polynomial $w$ (or $v$) has one zero  at  some point of $\mathbb{C}\mathbb{P}^1\setminus \{[1:+i],[1:-i]\}$.
This always happens, except for the  named {\em exceptional potentials} given by
$$V_{k,l}(q_1,q_2)=  \alpha (q_2-i q_1)^l (q_2+iq_1)^{k-l},  \; \mbox{where} \;   l=0,\dots , k  \; \mbox{and} \;  \alpha\in\mathbb{C}\setminus\{0\}.$$
Indeed,  it follows that these potentials lack Darboux points, see \cite{Maciejewski2005}.  So,  Morales-Ramis theory concerning homogeneous potentials cannot be applied to investigate their integrability.

The goal of this paper  is to give an step forward in the issue of existence of
polynomial first integrals  for two degrees of freedom complex Hamiltonian systems
with Hamiltonian function of the form
\begin{equation}\label{B1.1}H = \frac{1}{2}(p_1^2+p_2^2)+ V_{k,l}(q_1,q_2),\end{equation}
with exceptional homogeneous polynomial potentials of degree $k$
given by
\begin{equation}\label{pot_exep}
V_{k,l}(q_1,q_2)=   \alpha (q_2-i q_1)^l (q_2+iq_1)^{k-l} , \hspace{1.2cm}  l=0,1,\dots , k,
\end{equation}
where $k$ is a positive integer. For more details, see \cite{Maciejewski2005}.

As the Hamiltonian \eqref{B1.1} is time independent and the motion takes place on the submanifold
$\mathbb{C}^{2n}$ defined by a constant value $H$, which means that  $H$ itself is always a first integral, and since the Hamiltonian  \eqref{B1.1} has two degree of freedom,  our system will be completely (or Liouville) integrable  
if we can find a second first integral $F$ functionally independent of the Hamiltonian function $H,$
satisfying that $\{H,F\}=0$; that is, if $F$ and $H$ commute under the Poisson bracket.  For more information
the reader is referred to  \cite{Abraham}.

In  \cite{Hietarinta1983}, Hietarinta proved   that the Hamiltonian systems  with the exceptional potentials $V_{k,0}$, $V_{k,1}$,  $V_{k,k-1}$, $V_{k,k}$ and $V_{k,k/2}$   for even  $k$     are polynomial integrable, that is,
they have an additional polynomial first integral functionally  independent of the  Hamiltonian.
As a consequence,  all  exceptional potentials  for $k\leq 4$   are completely integrable with a polynomial integral besides the classical first integral. Additionally, Llibre and Valls in \cite{Llibre2015}  proved that  Hamiltonian systems  having potential  $V_{k,l}$ with even $k\geq 6$  and   $l= 2, \dots , k/2-1,k/2+1,\dots , k-2$,  do not admit an additional  polynomial first integral.
They  performed a non symplectic change of variable  to reduce the issue of existence of analytic first integrals in the Hamiltonian system  to an equivalent  weight-homogeneous polynomial differential system.
Also, the authors in \cite{Llibre2015}  claimed  that nothing is known about the integrability in the case when $k$ is odd,
which still remains  as an open problem.
Nevertheless,  Nakagawa et al.  \cite{Nakagawa2005}  used the so-called direct method to obtain  an additional polynomial first integral  of degree 4 in the momenta for the exceptional potential of degree seven with $l=2$, expressed  as
$V_{7,2}= (q_2-iq_1)^2(q_2+iq_1)^5$.      
Furthermore,  Nakagawa and  Yoshida \cite{Nakagawa2001} claimed that by using the direct method they have checked that for
$5 \leq k \leq 20$, $k\neq 7$, the Hamiltonian system with two degree of freedom and   potential $V_{k,l}$  
and $l\not\in\{0,1,k/2\}$, does not admit a further first polynomial integral with degree smaller than 9  in the momenta.

\section{Exceptional potentials viewed as  polynomial bi-homogeneous potentials}
In \cite{Combot2020}, Combot et al. used  a theorem by Moulin Ollagnier  about integrability of the homogeneous Lotka-Volterra system (see \cite{Moulin2001}),  that
allowed them to make a complete classification of integrable cases of Hamiltonian systems with two degrees
of freedom endowed with polynomial bi-homogeneous potentials. 
As a result, the necessary conditions for the integrability are transformed in a way that leads  to  investigate the integrability of a family of bi-homogeneous potentials which depend on two integer parameters.
This result inspired us to see an exceptional potential as a polynomial bi-homogeneous potential,  in order  to  recover
in a straight way the polynomial integrability of  Hamiltonian system with
potential $V_{k,l}$ with $l=0,1,k-1,k$ and $l=k/2$ for $k$ even.

Consider the symplectic change of coordinates  $x_1=q_1 + i q_2$,  $y_1=\frac{1}{2}( p_1 - i p_2)$,
$x_2= q_1 - i q_2$, $y_2=\frac{1}{2}(p_1 + i p_2)$.
Through some direct calculations, we obtain that  the Hamiltonian is transformed into
\begin{equation}\label{ham1}
H = 2 y_1y_2 +  \beta \; x_1^l x_2^{k-l},  
\end{equation}
where $\beta=  (-1)^l  i^k \alpha$, and the potential is given by
\begin{equation}\label{pot_hom}
 V_{k,l}=\beta \; x_1^l x_2^{k-l}, \end{equation}
which is   usually called  polynomial bi-homogeneous.
An application of the Combot et al. theorems formulated below, will allow us to deduce our result.

\begin{theorem}[Combot et al.  in \cite{Combot2020}]\label{combot_tma1}
The Hamiltonian system $H(x_1,x_2,y_1,y_2) = 2y_1y_2 + x_1^{k_1}x_2^{k_2}$ with $(k_1,k_2)\in \mathbb{Z}^2$
is rationally integrable if and only if either $(k_1,k_2)$ or $(k_2,k_1)$ belongs to the following list
\begin{equation*}
(0,k), \qquad (1,k), \qquad (k,k), \qquad (k,-2-k), \qquad (2,5), \qquad k\in \mathbb{Z}.
\end{equation*}
\end{theorem}

\begin{theorem}[Combot et al.  \cite{Combot2020}]\label{combot_tma2}
The Hamiltonian system given by $H(x_1,x_2,y_1,y_2) = 2y_1y_2 + x_1^{k_1}x_2^{k_2}$
 with $(k_1,k_2)\in \mathbb{Q}^2$ is   algebraically integrable if and only if
\begin{enumerate}
\item[1.] $(k_1,k_2)=(0,k)$, the additional first integral being $J=y_1$.
\item[2.] $(k_1,k_2)=(k,k)$,  \quad $J=y_1x_1-y_2x_2$.
\item[3.] $(k_1,k_2)=(-\frac{1}{2},k)$, \quad $J=y_1(y_1x_1-y_2x_2)-\frac{1}{2}x_1^{-1/2} x_2^{k+1}$.
\item[4.] $(k_1,k_2)=(1,k)$, \quad $J=y_1^2+\frac{1}{k+1} x_2^{k+1}$.
\item[5.] $(k_1,k_2)=(-k-2,k)$, \quad $J=(y_1x_1-y_2x_2)^2-2x_1^{-k-1}x_2^{k+1}$.
\item[6.] $(k_1,k_2)=(2,5)$,  
$J =6y_1^3(y_1x_1-y_2x_2)+x_2^6(4y_1^2x_1^2-8y_1 y_2 x_1x_2+y_2^2x_2^2)-x_1^3x_2^{12}$.
\item[7.] $(k_1,k_2)=(2,-7/4)$,  
$$J = y_1^6+8x_2^{-3/4} y_1^3(4y_2x_2-y_1x_1)+8x_2^{-3/2} (5y_1^2x_1^2-4y_1y_2x_1x_2+8y_2^2x_2^2)-48x_1^3x_2^{-9/4},$$
\item[8.]$(k_1,k_2)=(2,-10/7)$,  
$$J = 4y_1^6+14x_2^{-3/7}y_1^3(7y_2x_2-4y_1x_1)+343x_2^{-6/7} (y_1x_1-7y_2x_2)(y_1x_1-y_2x_2)-2058x_1^3x_2^{-9/7},$$
\item[9.] $(k_1,k_2)=(3,-9/5)$,  
$$J = 4y_1^4-10x_2^{-4/5} (3y_1^2 x_1^2-30y_1y_2x_1x_2 + 25 y_2^2x_2^2) + 225x_1^4x_2^{-8/5},$$
\item[10.] $(k_1,k_2)=(-2/3,-7/3)$,  
$$J = 8y_1(y_1x_1-y_2x_2)^3-x_1^{-2/3}x_2^{-4/3}(13y_1^2x_1^2-44y_1y_2x_1x_2+4y_2^2x_2^2)+16x_1^{-1/3}x_2^{-8/3},$$
\item[11.] $(k_1,k_2)=(-2/3,-10/3)$,  
$$J = 4y_1^2 (y_1x_1-y_2x_2)^4 +2x_1^{-2/3}x_2^{-7/3}y_1(-4y_1^3x_1^3+13y_1^2y_2x_1^2x_2 -20y_1y_2^2x_1x_2^2+2y_2^3x_2^3),$$
\item[12.] $(k_1,k_2)=(-2/3,-5/6)$,  
$$J = 2x_1^{-2/3} (y_1x_1-y_2x_2)^2[x_2^{1/6}-2y_1x_1^{2/3}(y_1x_1-y_2x_2)]+x_1^{-1/3}x_2^{1/3},$$
\item[13.] $(k_1,k_2)=(-3/4,-9/4)$,  
$$J = 4y_1 (y_1x_1-y_2x_2)^3-2x_1^{-3/4} x_2^{-5/4}(3y_1^2x_1^2-12y_1y_2x_1x_2+y_2^2x_2^2)+9x_1^{-1/2}x_2^{-5/2}.$$
\end{enumerate}
\end{theorem}

At this stage the reader should be warned that we are 
interested in studying polynomial potentials, as a consequence we will only consider the cases of items  from 1 to 6. 

Since  the potentials \eqref{pot_exep} and \eqref{pot_hom} are equivalent, we
 are thus led to the following consequence of Theorems   \ref{combot_tma1} and \ref{combot_tma2}.

\begin{proposition}\label{prop1}
The Hamiltonian system  with Hamiltonian function given by  \eqref{ham1} is  polynomially  integrable if and only if potentials and additional first integrals $J$ are given by
\begin{enumerate}
\item[1.] $V_{k,0}= \beta\: x_2^k$, \quad $\beta=i^k \alpha$, \quad $J=y_1$.
\item[2.] $V_{k,k}=\beta\: x_1^k$,  \quad $\beta=(- i)^k \alpha$,\quad $J=y_2$.
\item[3.] $V_{k,k/2}=  \beta x_1^{k/2} x_2^{k/2}$, \quad  with even $k$,  $\beta=\alpha$, \quad $J=y_1x_1-y_2x_2$.
\item[4.] $V_{k,1}= \beta \: x_1 x_2^{k-1}$, \quad $\beta=-i^k \alpha$,  \quad $J=y_1^2- \dfrac{1}{k} i^k   \alpha x_2^{k}$.
\item[5.] $V_{k,k-1} = \beta \: x_1^{k-1}x_2$, \quad $\beta =(-1)^{k-1} i^k \alpha$, \quad $J= y_2^2 + \dfrac{1}{k} (-1)^{k+1} i^k \alpha x_1^k $.
%
\item[6.] $V_{7,2}=  \beta x_1^2 x_2^5$, \quad $\beta=-i \alpha$, \quad $\alpha = i$, \\
$J=16y_1^3(y_1x_1-y_2x_2)+x_2^6(4y_1^2x_1^2-8y_1 y_2 x_1x_2+y_2^2x_2^2)-x_1^3x_2^{12}$.
\item[7.] $V_{7,5}=  \beta x_1^5 x_2^2$, \quad $\beta=i\alpha$, \quad $\alpha = -i$,  \\
$J=16y_2^3(y_2x_2-y_1x_1)+x_1^6(4y_2^2x_2^2-8y_2 y_1 x_1x_2+y_1^2x_1^2)-x_2^3x_1^{12}$.
\end{enumerate}
\end{proposition}

\begin{proof}
The function $J$ is a first integral of the  Hamiltonian system  with
Hamiltonian function  \eqref{ham1} given by  $$H(x_1,x_2,y_1,y_2) = 2 y_1y_2 +   V_{k,l} (x_1,x_2),$$ if
 the Poisson bracket  
$$\{H,J\} = \frac{\partial H}{\partial y_1} \frac{\partial J}{\partial x_1} -
 \frac{\partial H}{\partial x_1} \frac{\partial J}{\partial y_1} + \frac{\partial H}{\partial y_2} \frac{\partial J}{\partial x_2}
 - \frac{\partial H}{\partial x_2} \frac{\partial J}{\partial y_2}
$$ is zero. Thence, the proof of the statements  in the theorem relies on a direct construction of $\{H,J\}$
for each one-item in the list  given below.

\begin{enumerate}
\item[1.] $V_{k,0}= i^k \alpha\: x_2^k $,  \quad $J=y_1$
$$
\{H,J\}  = 2 y_2 \cdot 0 + 2y_1\cdot 0 +0\cdot 1 -  k i^k \alpha x_2{k-1}\cdot 0  =0.
$$
\item[2.] $V_{k,k}=(- i)^k \alpha\: x_1^k$,  \quad $J=y_2$
$$
\{H,J\}  = 2 y_2 \cdot 0 + (-1)^{k+1}  k i^k \alpha x_1{k-1}\cdot 0  + 2y_1\cdot 0  + 0\cdot 1 =0.
$$
\item[3.] $V_{k,k/2}=  \alpha x_1^{k/2} x_2^{k/2}$, \quad   \quad $J=y_1x_1-y_2x_2$
$$
\{H,J\}  = 2 y_2y_1 -\frac{1}{2}k\alpha x_1^{\frac{k}{2}-1} x_2^{k/2}\cdot x_1-2 y_1y_2 -\frac{1}{2} k\alpha x_1^{k/2} x_2^{\frac{k}{2}-1} \cdot(-x_2)=0.
$$
\item[4.] $V_{k,1}= -i^k \alpha \: x_1 x_2^{k-1}$,   \quad $J=y_1^2- \dfrac{1}{k} i^k   \alpha x_2^{k}$
\begin{eqnarray*}
\{H,J\}  &=& 2 y_2\cdot 0 + i^k\alpha x_2^{k-1} \cdot (2y_1) + 2y_1 \cdot (-\alpha i^k x_2^{k-1})
+ i^k (k-1) \alpha x_1x_2^{k-2}\cdot 0 \\
&=& 2 i^k\alpha x_2^{k-1} y_1 - 2 \alpha i^k x_2^{k-1}y_1=0
\end{eqnarray*}
\item[5.] $V_{k,k-1} = (-1)^{k-1} i^k \alpha \: x_1^{k-1}x_2$,  \quad $J= y_2^2 + \dfrac{1}{k} (-1)^{k+1} i^k \alpha x_1^k $
\begin{eqnarray*}
\{H,J\}  &=& 2y_2 \cdot (-1)^{k+1}i^k \alpha x_1^{k-1} + (-1)^ki^k(k-1)\alpha x_1^{k-2} x_2 \cdot 0 + 2y_1\cdot 0+
(-1)^ki^k\alpha x_1^{k-1} \cdot (2 y_2)\\
&=&0
\end{eqnarray*}
\item[6.] $V_{7,2}=   x_1^2 x_2^5$,  \\
\qquad $J=16y_1^3(y_1x_1-y_2x_2)+x_2^6(4y_1^2x_1^2-8y_1 y_2 x_1x_2+y_2^2x_2^2)-x_1^3x_2^{12}$
\begin{eqnarray*}
\{H,J\}  
&=& -6 x_1^2x_2^{12}y_2+16x_1x_2^6y_1^2y_2+32y_1^4y_2-16x_2^7y_1y_2^2 \\
&& \qquad- \;16x_1^3x_2^{11}y_1-128x_1^2x_2^5y_1^3+16x_1^2x_2^{12}y_2+96x_1x_2^6 y_1^2y_2 \\
&& \qquad - \; 24x_1^3x_2^{11}y_1+48x_1^2x_2^5y_1^3-112x_1x_2^6y_1^2y_2-32y_1^4y_2+16x_2^7y_1y_2^2\\
&& \qquad  + \; 40x_1^3x_1^{11} y_1+80x_1^2x_2^5y_1^3-10x_1^2x_2^{12}y_2 = 0 
\end{eqnarray*}
\item[7.] $V_{7,5}=  x_1^5 x_2^2$,   \\
\qquad $J=16y_2^3(y_2x_2-y_1x_1)+x_1^6(4y_2^2x_2^2-8y_2 y_1 x_1x_2+y_1^2x_1^2)-x_2^3x_1^{12}$
\begin{eqnarray*}
\{H,J\}  &=& 2 y_2 \cdot (-4)(3x_1^{11}x_2^3-2x_1^7y_1^2+14x_1^6x_2y_1y_2-6x_1^5x_2^2y_2^2+4y_1y_2^3\\
&& \qquad - \; 5x_1^4x_2^2 \cdot  2(x_1^8y_1-4x_1^7x_2y_2-8x_1y_2^3) \\
&& \qquad  + \; 2y_1\cdot (-3x_1^{12}x_2^2-8x_1^7y_1y_2+8x_1^6x_2y_2^2+16y_2^4) \\
&& \qquad  - \; 2x_1^5 x_2 \cdot 8(-x_1^7x_2y_1+x_1^6x_2^2y_2-6x_1y_1y_2^2+8x_2y_2^3) =0 
\end{eqnarray*}
\end{enumerate}
\end{proof}

\begin{remark}
We can observe that $V_{7,2}(x_1,x_2,y_1,y_2)=V_{7,5}(x_2,x_1,y_2,y_1)$. That is, there exists a permutation $\theta\in S_4$ such that $V_{7,5}(x_1,x_2,y_1,y_2)=V_{7,2}(\theta(x_1),\theta(x_2),\theta(y_1),\theta(y_2))$, where the permutation $\theta$ and permutation matrix $A_\theta$ are given respectively by $$\theta=\begin{pmatrix}
1&2&3&4\\ 2&1&4&3
\end{pmatrix}=(1,2)(3,4),\quad A_\theta=\begin{pmatrix}
0&1&0&0\\
1&0&0&0\\
0&0&0&1\\
0&0&1&0
\end{pmatrix}$$ and $x_1\mapsto 1$, $x_2\mapsto 2$, $y_1\mapsto 3$ and $y_2\mapsto 4$. Now, setting the linear transformations $T(q_1,q_2,p_1,p_2)=\frac{1}{2}(q_1+iq_2,q_1-iq_2,p_1-ip_2,p_1+ip_2)$ and $T_\theta$ with companion matrix $A_\theta$, we obtain $T_\theta\circ T\circ T_\theta^{-1}(q_1,q_2,p_1,p_2)=(q_2,q_1,p_2,p_1)$. Thus, we can see that in coordinates $(q_1,q_2,p_1,p_2)$ it is satisfied that $V_{7,5}(q_1,q_2,p_1,p_2)= V_{7,5}(q_2,q_1,p_2,p_2)$.
\end{remark}

Finally, we summarize all the result of the Proposition \ref{prop1} in terms of the variables  $(q_1,q_2,p_1,p_2)$  in Table \ref{tab1}.
\begin{table}[h]
\caption{Exceptional potential $V_{k,l}$    and the additional  first integral $J$ in terms of their natural variables $(q_1,q_2,p_1,p_2)$.}
 \label{tab1}
\centering
{\scriptsize
\begin{tabular}{ c | c  }
\hline
Exceptional potential  $V_{k,l}$ & First integral $J$\\
\hline
$V_{k,0} = \alpha (q_2+ i q_1)^k$ & $\frac{1}{2} (p_1-ip_2)$\\
$V_{k,k}=\alpha (q_2- i q_1)^k$ & $\frac{1}{2} (p_1 + ip_2)$\\
$V_{k,k/2}=  \alpha (q_2-iq_1)^{k/2}(q_2+iq_1)^{k/2}$ & $-\frac{i}{2}(p_1q_2 - q_1p_2)$\\
$V_{k,1}= \alpha \: (q_2-i q_1)(q_2+iq_1)^{k-1}$ & $\frac{1}{4}(p_1-ip_2)^2- \dfrac{1}{k} i^k   \alpha  (q_1-iq_2)^{k}$\\
$V_{k,k-1} = \alpha (q_2- iq_1)^{k-1}(q_2+i q_1)$ & $\frac{1}{4}(p_1+ip_2)^2 + \dfrac{1}{k} (-1)^{k+1} i^k \alpha (q_1+iq_2)^k $\\
& \\
$V_{7,2}=  (q_1-iq_2)^5 (q_1+iq_2)^2$ &  $-(p_1-ip_2)^3 (p_1+ip_2)(q_1-iq_2) +\frac{1}{4} (p_1+ip_2)^2 (q_1-iq_2)^8 $\\
 & $+ (p_1-ip_2)^4 (q_1+i q_2) - 2 (p_1-i p_2) (p_1+i p_2) (q_1-iq_2)^7 (q_1 + i q_2)$  \\
 & $+ (p_1-ip_2)^2 (q_1-iq_2)^6 (q_1+i q_2)^2 - (q_1-iq_2)^{12} (q_1+i q_2)^3 $ \\
& \\
$V_{7,5}= (q_1-iq_2)^2(q_1+iq_2)^5$ &  $ (p_1+ip_2)^4(q_1-iq_2)-(p_1-ip_2)(p_1+ip_2)^3(q_1+iq_2)$ \\
 & $ +(p_1 + i p_2)^2(q_1-iq_2)^2 (q_1+iq_2)^6$ \\
 & $- 2 (p_1-ip_2)(p_1+ip_2)(q_1-iq_2)(q_1+iq_2)^7$ \\
& $ + \frac{1}{4} (p_1-i p_2)^2 (q_1+i q_2)^8 - (q_1-iq_2)^3 (q_1+iq_2)^{12}$\\
\end{tabular}}
\end{table}

Let us conclude by saying that our results agree well with those of Hietarinta \cite{Hietarinta1987} and
Nakagawa et al. \cite{Nakagawa2005}. Although it is obvious that $V_{7,5}(x_1,x_2,y_1,y_2)=V_{7,2}(x_2,x_1,y_2,y_1)$, it is not evident that $V_{7,5}(q_1,q_2,p_1,p_2)=V_{7,2}(q_2,q_1,p_2,p_1)$ due to we need the conjugation of $T$ with $T_\theta$. Moreover, as far as the authors  are aware, this is the first time that an additional polynomial first integral for $V_{7,5}(q_1,q_2,p_1,p_2)$, different from the Hamiltonian one, is given explicitly.
\bigskip
\section{Algebraic Analysis}
We start this section summarizing for reader's convenience the major facts on algebraic aspects, invariant planes and differential Galois group of the first variational equation. For more information on this subject the reader is referred to \cite{aas2018,ams2021,aabd2013,aad2009,almp2018,ay2020}.

The starting point is the concept of differential field. A differential field $K$ is a field endowed with a derivation $\partial$, such that for all $a,b\in K$ it satisfied:
\begin{enumerate}
    \item $\partial(a+b)=\partial a+\partial b$
    \item $\partial (a\cdot b)=a\cdot \partial b+\partial a\cdot b$
    \item $\partial \left(\dfrac{a}{b}\right)=\dfrac{1}{b^2}(a\cdot \partial b-\partial a\cdot b)$.
\end{enumerate}
The field of constants of $K$, denoted by $C_K$, is given by $$C_K=\{c\in K:\,\, \partial (c)=0\}.$$
Consider the system of linear differential equations over $K$ given by
\begin{equation}\label{ds1}
    \frac{d}{dx}\begin{pmatrix}
    y_1\\
    y_2\\
    \vdots\\
    y_n
    \end{pmatrix}=\begin{pmatrix}
    a_{11}&a_{12}&\ldots&a_{1n}\\
    a_{21}&a_{22}&\ldots&a_{2n}\\
    \vdots&&&\\
    a_{n1}&a_{n2}&\ldots&a_{nn}\\
    \end{pmatrix}\begin{pmatrix}
    y_1\\
    y_2\\
    \vdots\\
    y_n
    \end{pmatrix},\quad a_{ij}=a_{ij}(x)\in K,\quad y_i=y_i(x)
\end{equation}
The fundamental matrix $U=[u_{ij}]$ of the system \eqref{ds1} plays an important role in this theory because any differential extension of $L/K$ must contains $K$ and $u_{ij}$, i.e., $L=K\langle u_{ij}\rangle$. The Picard-Vessiot extension $L/K$ is the extension of $K$ preserving the field of constants, that is $C_L=C_K$. Thus, given a system of first order linear differential equations \eqref{ds1}, the differential Galois group of $Y'=AY$, denoted by $DGal(L/K)$, is the group of $K$-differential automorphisms from $L$ to $L$, that is, $\sigma:\,L\mapsto L$, $\partial(\sigma (a))=\sigma(\partial a)$,  $$DGal(L/K)=\{\sigma:\,\, \sigma(a)=a,\,  \forall a\in K\}.$$ 
The connected identity component of $DGal(L/K)$, denoted by $(DGal(L/K))^0$ is the biggest algebraic subgroup of $DGal(L/K)$ containing the identity. In a general framework, more that Hamiltonian systems with homogeneous potentials, Morales-Ramis theory is the theory that relates differential Galois theory with the integrability of dynamical systems. In particular, Morales-Ramis Theorem for Hamiltonian systems says that \emph{if a Hamiltonian system is integrable, then the connected identity component of the differential Galois group of the first variational equation along any particular solution is an abelian group}.\\

In this way, an \emph{invariant plane} of the Hamiltonian system with first integral \eqref{B1.1} is a restricted plane  satisfying the motion equations. The solutions of motion equations restricted to the invariant plane are called \emph{integral particular curves}, also known \emph{particular solutions} of the Hamiltonian system. More precisely, we are interested in the invariant plane $\Gamma=\{q_2=p_2=0\}$ with particular solution $\gamma(t)=(q_1(t),0,p_1(t),0)$. Let be $q_1=x$ and $-V'(q_1,0)=f(x)$. The explicit solutions for differential equation
\begin{equation}\label{lmp1}
    \frac{d^2x}{dt^2}=f(x)
\end{equation}
are related with the integral curve $(x,\dot x)$ of the one degree of freedom Hamiltonian system
\begin{equation}\label{lmp2}
\dot x=y,\quad  \dot y=f(x), 
\end{equation}
where the energy  is given by
\begin{equation}\label{lmp2a}
h=\frac{y^2}{2}-\int_{x_0}^xf(\tau)d\tau.
\end{equation}

The first variational equation of the Hamiltonian system \eqref{ham_2n}, along a particular integral curve $\gamma(t)$ is given by  the linear system 
\begin{equation}\label{vareq}
    \dot{\mathbf{\xi}}=A(t)\mathbf{\xi},\qquad \mathbf{\xi}=
    \begin{pmatrix}
    \xi_1\\ \xi_2\\ \xi_3\\ \xi_4
    \end{pmatrix},
    \qquad A(t)=X'_H \Big|_{\gamma(t)},\qquad X_H=J_2\nabla H,
\end{equation} where $X_H$ depends on $q_1,q_2,p_1,p_2$ and it is restricted to $\gamma(t)$; $\xi$ is unknown vector which depends on $t$ and $J_2$ is the standard symplectic matrix given by
$$J_2=\begin{pmatrix}
0&0&1&0\\
0&0&0&1\\
-1&0&0&0\\
0&-1&0&0
\end{pmatrix}$$

Taking into account the concepts given above, we have the following result.
\begin{proposition}\label{propgal1}
Let $k$ be an even number.
The Hamiltonian system \eqref{ham_2n} with exceptional polynomial potential $V_{k,l}$ has infinite invariant planes \begin{equation}\label{eqgam}\Gamma_{\lambda_1,\lambda_2}=\left\{(q_1,q_2,p_1,p_2)\in \mathbb{C}^4:\, \lambda_1q_1+\lambda_2q_2=\lambda_1p_1+\lambda_2p_2=0,\,\, \lambda_1,\lambda_2\in\mathbb{C}^*\right\}\end{equation} if and only if $k=2l$ or $\lambda_1/\lambda_2=\pm i$.
\end{proposition}
\begin{proof}
The system   \eqref{ham_2n}  with potential $V_{k,l}$ goes over to 
\begin{equation}\label{eqq_h2}
\begin{array}{l}
   \dot q_1=  p_1,  \\
  \dot q_2 =  p_2,\\
 \dot p_1=  \alpha(q_2 + iq_1)^{k - l - 1}(q_2 - iq_1)^{l - 1}(iq_2(2l-k) -kq_1), \vspace{0.1cm}\\
 \dot p_2=  \alpha(q_2 + iq_1)^{k - l - 1}(q_2 - iq_1)^{l - 1}(iq_1(k - 2l) - kq_2).
\end{array}
\end{equation}
Assume $\lambda_1\lambda_2\neq 0$. The conditions $\lambda_1q_1+\lambda_2q_2=\lambda_1p_1+\lambda_2p_2=0$ reads $q_2=cq_1$ and $p_2=cp_1$, where $c=-\lambda_1/\lambda_2\neq 0$. Since  $(q_1,cq_1,p_1,cp_1)$ are invariant planes,  it is satisfied that $\dot q_2=c\dot q_1$ and $\dot p_2=c p_1$ in Eq. \eqref{eqq_h2}. Thus, $\dot p_2/\dot p_1$ in Eq. \eqref{eqq_h2} reads as
$$c=\frac{iq_1(k - 2l) - kq_2}{iq_2(2l-k) - kq_1}=\frac{iq_1(k - 2l) - ckq_1}{icq_1(2l-k) - kq_1}=\frac{i(k - 2l) - ck}{ic(2l-k) - k}.$$
Solving the previous expression we obtain $(c^2+1)(2l-k)=0$ and due to $c\neq 0$ we conclude $c=\pm i$ or $k=2l$.

Next, we show that $(q_1,cq_1,p_1,cp_1)$ is a family of invariant planes.
To see this, we consider  $c=\pm i$ in Eq. \eqref{eqq_h2} and trivially $(q_1,\pm iq_1,p_1,\pm ip_1)$ are invariant planes. Now, we consider $k=2l$. Then  $$V_{2l,l}=\alpha (q_1^2+q_2^2)^l.$$ 
Since  $k=2l$ is an even number,  $l\in\mathbb{Z}^+$ to avoid the case of constant polynomial. Hence, the Hamiltonian system associated with this potential, Eq. \eqref{eqq_h2}, becomes
\begin{equation}\label{eq_h1}
\begin{array}{l}
  \dot q_1= p_1, \\
  \dot q_2= p_2,\\
 \dot p_1=-2l\alpha q_1(q_1^2+q_2^2)^{l-1},\vspace{0.1cm}\\
 \dot p_2=-2l\alpha q_2(q_1^2+q_2^2)^{l-1}.
\end{array}
\end{equation}

A straightforward analysis of Eq. \eqref{eq_h1} shows that $\dfrac{dq_2}{dq_1}=\dfrac{p_2}{p_1}$ and $\dfrac{dp_2}{dp_1}=\dfrac{q_2}{q_1}$ imply  $\dfrac{q_2}{q_1}=\dfrac{p_2}{p_1}=c\in\mathbb{C}^*$. Hence
$c=-\lambda_2/\lambda_1$, $\lambda_1q_1+\lambda_2q_2=\lambda_1p_1+\lambda_2p_2=0$ and therefore $\Gamma_{\lambda_1,\lambda_2}$ defined in Eq. \eqref{eqgam} are infinite invariant planes of Eq. \eqref{eqq_h2}, which completes the proof.

\end{proof}

Assuming $\lambda_1\lambda_2=0$ and $|\lambda_1|+|\lambda_2|>0$, the reader should be warned that $(0,q_2,0,p_2)$ and $(q_1,0,p_1,0)$ are particular cases of  invariant planes $\Gamma_{\lambda_1,\lambda_2}$, defined in Eq. \eqref{eqgam}, of the system given in Eq.  \eqref{eqq_h2} for $k=2l$.  We recall that $H|_\Gamma$ is the hamiltonian restricted to the invariant plane $\Gamma$. The following result is an application of Morales-Ramis Theory to an integrable hamiltonian system.

\begin{proposition}\label{propgal2}
The connected identity component of the differential Galois group of the first variational equation of the Hamiltonian system \eqref{eq_h1}, along particular solutions with $H|_\Gamma=0$,
is a subgroup of the diagonal group.
\end{proposition}

\begin{proof}
The hamiltonian system \eqref{eq_h1} reads
\begin{equation*}
\begin{array}{l}
  \dot q_1= p_1, \\
  \dot q_2= p_2,\\
 \dot p_1=-2l\alpha q_1(q_1^2+q_2^2)^{l-1},\vspace{0.1cm}\\
 \dot p_2=-2l\alpha q_2(q_1^2+q_2^2)^{l-1}.
\end{array}
\end{equation*}
A particular solution over the invariant plane $\Gamma=\{(q_1,cq_1,p_1,cp_1)\}$ of Eq. \eqref{eq_h1} is obtained through the solution of the  differential equation $\dot p_1=\ddot q_1= \mu q_1^{2l-1}$, where $\mu=-2l\alpha (c^2 +1)^{l-1}$. Setting $x=q_1$, we have the differential equation $\ddot x=\mu x^{l}$, which corresponds to Eq. \eqref{lmp1} with $f(x)=\mu x^{2l-1}$. Indeed, from   \eqref{lmp2a}  we have
$$H|_\Gamma=h=\frac{(\dot x)^2}{2}-2\nu^2 x^{2l},\quad -2\nu^2=\frac{\mu}{l+1}.$$ 
Owing to $H|_\Gamma=h=0$, we obtain the particular solutions
$$x_\pm(t)=\frac{1}{\sqrt[l-1]{b_1\pm b_2 t}},\qquad b_1=c_1(l-1),\qquad b_2=2\nu(l-1).$$
At this stage we need to obtain the first variational equation along these particular solutions. It follows that $$ \dot{\mathbf{\xi}}=A(t)\mathbf{\xi},\quad \mathbf{\xi}=
    \begin{pmatrix}
    \xi_1(t)\\\xi_2(t)\\ \xi_3(t)\\ \xi_4(t)
    \end{pmatrix},
    \quad A(t)=\begin{pmatrix}
    0&0&1&0\\
    0&0&0&1\\
    \widehat{\mu}_1x(t)^{2l-2}&\widehat{\mu}_2x(t)^{2l-2}&0&0\\
    \widehat{\mu}_2x(t)^{2l-2}&\widehat{\mu}_3x(t)^{2l-2}&0&0
    \end{pmatrix},$$ where  $\widehat{\mu}_1=-2l\alpha(c^2+2l-1)(c^2+1)^{l-2}$, $\widehat{\mu}_2=-4l\alpha c(l-1)(c^2+1)^{l-2}$ and $\widehat{\mu}_3=-2l\alpha(2lc^2 - c^2 + 1)(c^2+1)^{l-2}$

Since  the particular solutions $x_\pm(t)$ are algebraic of degree $l-1$, it follows that $$x_\pm(t)^{2l-2}=\frac{1}{\left(b_1\pm b_2t\right)^2},$$ and therefore the coefficients matrix of the variational equation has rational elements. Now, this variational equation has two regular singularities $t=-b_1/b_2$ and $t=\infty$, which are of regular type. Moreover, due to $\widehat{\mu}_1\neq \widehat{\mu}_2\neq \widehat{\mu}_3$ we cannot obtain logarithmic solutions and the general solution has the form $c_1\xi_1(t)+c_2\xi_2(t)+c_3\xi_3(t)+c_4\xi_4(t)$, $\xi_k(t)=g_k(t)^{r_k}$, where $g_k\in \mathbb{C}[t]$ and $r_k\in \mathbb{C}$. Thus, the logarithmic derivative of $\xi_k(t)$ is a rational function and then the  connected identity component of the differential Galois group is contained in the diagonal group.
\end{proof}
We observe that considering $h\neq 0$ involves harmonic oscillator ($l=1)$ and hyperelliptic functions ($l>1)$. For $l=1$ we fall in the case of classical harmonic oscillator, which is integrable and the variational equations have constant coefficients, for instance, the differential Galois group is isomorphic to the multiplicative group, which is a connected group. We cannot obtain elliptic functions because the exponent is always even ($2l$). At present we only remark that more information on the relation between integrability of Hamiltonian systems and solvability of variational equations,
the reader is addressed to  \cite{mora1,yo1,yo2}.

\section*{Acknowledgements}
M. Alvarez-Ram\'{\i}rez  has been partly supported from Prodep/M\'exico  grant number 511-6/2019-15963 and the grant  Sistemas Hamiltonianos, Mec\'anica y Geometr\'{\i}a  PAPDI2021 CBI-UAMI. P. Acosta-Hum\'anez has been partly supported from Internal Faculty of Sciences Founds Project No-integrabilidad de campos vectoriales polinomiales en el plano
complejo mediante la Teoría de Galois Diferencial, Instituto de Matem\'atica. Universidad Aut\'onoma de Santo Domingo.

The authors thank to J. J. Morales-Ruiz for his  constructive  comments on the first version of the paper. The authors are indebted with the anonymous referee by their useful comments and suggestions, which we have taken into consideration to improve the manuscript.

\bibliographystyle{plain}
\bibliography{acalst_final}

\end{document}